\documentclass[12pt]{amsart}
\usepackage{amsmath,amssymb,amsthm,amsfonts,setspace,hyperref,dsfont}

\newcommand{\NN}{\mathbb{N}}
\newcommand{\RR}{\mathbb{R}}
\newcommand{\CC}{\mathbb{C}}

\newcommand{\ZZ}{\mathbb{Z}}
\newcommand{\KK}{\mathbb{K}}

\newcommand{\abs}[1]{\lvert#1\rvert}

\newtheorem{theorem}{Theorem}
\newtheorem{corollary}{Corollary}[section]

\newtheorem{lemma}{Lemma}[section]
\newtheorem{example}{Exmaple}[section]
\newtheorem{proposition}{Proposition}[section]

\theoremstyle{definition}
\newtheorem{definition}{Definition}[section]
\newtheorem{remark}{Remark}[section]

\numberwithin{equation}{section}

\DeclareMathOperator{\diag}{diag}
\begin{document}

\title[DIFFERENTIAL RIGIDITY]{NEW CASES OF DIFFERENTIAL RIGIDITY FOR NON-GENERIC PARTIALLY
HYPERBOLIC ACTIONS}
\author{Zhenqi WANG}
\address{Department of Mathematics\\ The Pennsylvania State
University\\ University Park, PA,16802} \email{wang\_z@math.psu.edu}

\maketitle
\begin{abstract} We prove the locally differentiable rigidity of generic partially
hyperbolic abelian algebraic high-rank actions on compact
homogeneous spaces obtained from split symplectic Lie groups. We
also gave a non-generic action rigidity example on compact
homogeneous spaces obtained from $SL(2n,\RR)$ or $SL(2n,\CC)$. The
conclusions are based on geometric Katok-Damjanovic way and progress
towards computations of the generating relations in these groups.
\end{abstract}

\section{Introduction}
 Let
 $G$ be a $\RR$-semisimple Lie group with real rank greater than $2$,
 $\mathfrak{h}=\RR^k$ its split Cartan subalgebra and $\Gamma$ be a cocompact lattice in
 $G$. Let $A$ be a maximal split Cartan subgroup of $G$ with Lie algebra $\mathfrak{h}$, and $K$
be the compact part of the centralizer of $A$ which intersects with
$A$ trivially. Let $\Phi$ be the root system of $G$ with respect to
 $\mathfrak{h}$. Every root $r\in\Phi$ defines a Lyapunov hyperplane $\mathbb{H}_r=\ker r$. A 2-dimensional plane in
 $\mathfrak{h}$ is said to in general position if
it intersects each two distinct Lyapunov hyperplanes along distinct
lines. Let $\mathfrak{s}\subseteq \mathfrak{h}$ and $S$ the
connected subgroup in $G$ with Lie algebra $\mathfrak{s}$. The
action $\alpha_{0,S}$ of $S$ by left translations on $X:=G/\Gamma$
is generic if it contains a lattice contained in a 2-plane in
general position.

A. Katok and R.Spatzier proved the locally differentiable rigidity
of the Anosov full Cartan actions (also called Wely chamber flow) on
$K\backslash G/\Gamma$ by a harmonic analysis
method.\cite{Spatzier}. Later A. Katok and Damjanovic
\cite{Damjanovic2},\cite{Damjanovic3} proved the locally
differentiable rigidity of partially hyperbolic generic actions on
$X$ if $G$ is simple with $\Phi$ of nonsymplectic type with
combination of geometric methods and K-theory. The natural
difficulty of symplectic type is related to infinite types of
reducible Lyapunov-foliation cycles resulted by infinitly different
homotopic classes. For quasi-split groups, $(BC)_n$-type root
systems have the same infiniteness problem although the generating
relations are available\cite{Deodhar}. In this paper, we proved
locally differentiable rigidity of split symplectic Lie groups which
has been left open in \cite{Damjanovic3}. In fact, we can extend the
generic action rigidity results to quai-split groups(see remark
\ref{re:1}).

A  necessary condition for applicability of the Damjanovic--Katok
geometric method  (although not for local rigidity) is that
contracting distributions of various action elements and their
brackets of all orders   generate  the tangent space to the phase
space. Generic restrictions  for Cartan actions satisfy that
condition. Naturally one may look at non-generic restrictions of
Cartan actions.
\begin{example}
In $SL(4,\RR)$, consider the plane $P$ given by the equation $t_2 +
2t_3 + 3t_4 = 0.$ It is not in general position since the
intersection of Lyapunov hyperplanes obtained by roots $L_1-L_4$
with $P$ is the same as that of $L_2-L_3$. Though two unipotents
$v_{14}$ and $v_{23}$ are not both stable for any element of the
action, but consider
$$[v_{14}(t),v_{23}(s)]=[[v_{13}(t),v_{34}(1)],v_{23}(s)].$$
$v_{34}$ and $v_{23}$ are in the stable foliation of $(5,-8,1,2)$,
$v_{13}$ and $v_{23}$ in the stable foliation of $(0,-1,2,-1)$.
\end{example}
Hence we still get cocycle rigidity for action $\alpha_{0,P}$(more
details can be found in \cite{Damjanovic1}). Generally speaking, for
cocycle rigidity, ``generic'' is not necessary since we have enough
elements to trivialize all Lyapunov-cycles if $P$ intersects each
two distinct Lyapunov hyperplanes defined by simple roots along
distinct lines. It is based on the fact that Lie brackets of simple
roots and their inverse can generate the whole root system. But for
differential rigidity, more arguments are needed since usually the
coarse Lyapunov spaces are changed.

In this paper, we obtain an important example of locally
differentiable rigidity of non-generic actions on compact
homogeneous spaces obtained from $SL(2n,\RR)$ and $SL(2n,\CC)$. In
this example, the plane intersects Lyapunov hyperplanes defined by
simple roots along same lines, which disables the geometric
Katok-Damjanovic method. We gave new generating relations adapted to
the dynamical systems. These results are of independent interest and
have widen classes of locally rigid actions and the method can be
applied to other non-generic actions.

 To prove Theorem \ref{th:1} and \ref{th:4}, we make sufficient
progress towards the computations of generating relations of
$Sp(2n,\RR)$, $SL(2n,\RR)$ and $SL(2n,\CC)$ where $n\geq 2$.
Generating relations are available for these groups\cite{Matsumoto},
however, they do not provide sufficient information adapted to the
dynamical systems and need to be supplemented by more detailed
calculations.

Results for non-generic differential rigidity in other high rank
groups and more general conditions admitting differential rigidity
phenomenons will appear in a separate paper.

I'd like to thank my advisor Anatole Katok, who introduced me to the
non-generic problem and constantly encouraged me.
\subsection{The main results}
\begin{theorem}\label{th:1}  Let $\alpha_{0,S}$ be a high rank
generic restriction of the action of a maximal split Cartan subgroup
on $Sp(2n,\RR)/\Gamma$ where $n\geq 2$. If $\tilde{\alpha}$ is
$C^{\infty}$ action sufficiently $C^2$-close to $\alpha_{0,S}$, then
there exists a $C^\infty$ diffeomorphism $h : X \rightarrow X$ such
that $h^{-1}\tilde{\alpha}(a, h(x)) = \alpha_{0,\tilde{S}}$ where
$\tilde{S}$ is isomorphic and close to $S$ in a maximal split Cartan
subgroup of $Sp(2n,\RR)$.
\end{theorem}
 Note that for $Sp(2n,\RR)$ generating relations are available
in \cite{Matsumoto}, but to get enough information for reducible
classes, further calculations are carried out on the Schur
multiplier. The proof resembles Theorem 2 in \cite{Zhenqi} on
dealing with infinite homotopic classes.
\begin{remark}\label{re:1}
By \cite{Deodhar}, any quasi-spilt simple groups of non-$(BC)_n$
type and non-symplectic type, are subject to following generating
relations:
\begin{align*}
&(1) c(s,tz)=c(s,t)c(s,z),\qquad (2)c(tz,s)=c(t,s)c(z,s),\\
&(3) c(s,1-s)=1.
\end{align*}
Thus for full maximal Cartan actions of these groups, rigidity are
abstained by using almost the same manners as in \cite{Damjanovic3}.
For $(BC)_n$-type are groups $SU(m+1,m)$; for symplectic type are
$Sp(2n,\RR)$ and $SU(m,m)$. Rigidity of $SU(m+1,m)$ and $SU(m,m)$
are solved in \cite{Zhenqi}; rigidity of $Sp(2n,\RR)$ was proved in
Theorem \ref{th:1}. Hence we in fact have locally differential
rigidity for high rank quasi-split Lie groups.
\end{remark}
For a more general case when the actions are not generic, for
example, if $G=SL(2n,\RR)$($SL(2n,\CC)(n\geq 2)$ and $\Gamma$ a
cocompact lattice in $G$. Let
\begin{align*}
D_+=&\exp\mathbb{D_+}=\{\diag\bigl(\exp t_1,\dots,\exp t_n,\exp
(-t_1),\dots,\exp (-t_n)\bigl):\\
&(t_1,\dots,t_n)\in\RR^n\}.
\end{align*}
Let the root system with respect to $D_+$ be $\Phi$.
\begin{theorem}[Differential rigidity of non-generic actions ]\label{th:4} Let $G=SL(2n,\RR)$($SL(2n,\CC)$, $n\geq 2$).
If $\tilde{\alpha}$ is $C^{\infty}$ action sufficiently $C^2$-close
to $\alpha_{0,S}$ where $S$ contains a lattice contained in a
generic 2-plane in $\mathbb{D_+}$ with respect to $\Phi$. Then there
exists a $C^\infty$ diffeomorphism $h : X \rightarrow X$ such that
$h^{-1}\tilde{\alpha}(a, h(x)) = \alpha_{0,\tilde{S}}$ where
$\tilde{S}$ is isomorphic and close to $S$ in the centralizer of a
maximal split Cartan subgroup of $SL(2n,\RR)$($SL(2n,\CC)$).
\end{theorem}
\begin{remark} For any 2-plane $P$ in $D_+$ it is not generic with respect to the root system defined by the maximal Cartan
subgroups(that is $\{L_k-L_\ell\}_{k\neq \ell}$) since for any
different indices $i,j$, Lyapunov hyperplanes $\mathbb{H}_{i,j}$ and
$\mathbb{H}_{i\pm n,j\pm n}$ intersect $P$ on same lines. $\Phi$ is
different from the usual roots systems of special linear groups. It
behaves similarily to that of symplectic groups instead.
\end{remark}

\subsection{ Generating relations and Steinberg symbols} In this section
we state two theorems which play a crucial role in proofs of Theorem
\ref{th:1} and \ref{th:4}. The proof of those theorems are given in
Section \ref{sec:6} and \ref{sec:3} which comprise the algebraic
part of the paper.

We use $e_{k,\ell}$ to denote the $2n\times 2n$ matrix in with the
 $(k,\ell)$ entry  equal to 1, and all other entries equal to  0.
 Let
\begin{align*}
&f_{L_i+L_j}=(e_{i,j+n}-e_{j,i+n})_{i<j},\qquad
f_{L_i-L_j}=(e_{i,j}-e_{j+n,i+n})_{i\neq j},\\
&f_{-L_i-L_j}=(e_{j+n,i}-e_{i+n,j})_{i<j},\qquad
f_{2L_i}=e_{i,i+n},\\
&f_{-2L_i}=e_{i+n,n}.
\end{align*}
Let $\exp$ be the exponentiation map for matrices. For $t\in \RR$ we
write
\begin{align*}
&x_r(t)=\exp (tf_{r})\qquad\text{ for }0\neq r=\pm L_i\pm L_j.
\end{align*}
Then we have following results
\begin{theorem}\label{th:3}
$Sp(2n,\RR)$, $n\geq 2$ is generated by  $x_{r}(a)$, where $0\neq
r\in\Phi=\{\pm L_i\pm L_j\}(0\leq i,j\leq n)$  subject to the
relations:
\begin{align}
&x_r(a)x_r(b)=F_r (a+b),\label{for:26}\\
&[x_{r} (a), x_p (b)]=\prod_{ir+jp\in \Phi, i,j>0}
x_{ir+jp}(g_{ijpr}(a,b)), r+p\neq 0\label{for:27}\\
&[x_{r} (a), x_p (b)]=\emph{id}, \qquad 0\neq r+p\notin \pm L_i\pm
L_j,\label{for:5}
\end{align}
here $a,b\in\RR^*$ and $g_{ijpr}$ are functions of $a,b$ depending
only on the structure of $Sp(2n,\RR)$;

\begin{align}
&h_{L_1-L_2}(a)h_{L_1-L_2}(b)=h_{L_1-L_2}(ab),\label{for:6}
\end{align}
where
$h_{L_1-L_2}(t)=x_{L_1-L_2}(t)x_{L_2-L_1}(-t^{-1})x_{L_1-L_2}(t)x_{L_1-L_2}(-1)x_{L_2-L_1}(1)x_{L_1-L_2}(-1)$
for each $t\in \RR^{*}$;
\begin{align}\label{for:22}
h_{2L_{n}}(-1)h_{2L_{n}}(-1)=\emph{id},
\end{align}
where
\begin{align*}
h_{2L_{n}}(-1)&=\bigl(x_{2L_n}(-1)x_{2L_n}(1)x_{2L_n}(-1)\bigl)^2.
\end{align*}
\end{theorem}
Now we state the theorem about new generating relations in
$SL(2n,\KK)$, $\KK=\RR$ or $\CC$, $n\geq 2$. Let
\begin{align*}
&f_{L_i+L_j}(t_1,t_2)=(t_1 e_{i,j+n}+t_2e_{j,i+n})_{i<j},\qquad
f_{L_i-L_j}(t_1,t_2)=(t_1e_{i,j}+t_2e_{j+n,i+n})_{i\neq j},\\
&f_{-L_i-L_j}(t_1,t_2)=(t_1e_{j+n,i}+t_2e_{i+n,j})_{i<j},\qquad
f_{2L_i}(t)=te_{i,i+n},\\
&f_{-2L_i}(t)=te_{i+n,n}.
\end{align*}
For $(t_1,t_2)\in \KK^2$, $t\in\KK$ we write
\begin{align*}
&x_\rho(t)=\exp (tf_{\rho})\qquad\text{ for }\rho=\pm 2L_i,\\
&x_r(t_1,t_2)=\exp (f_r(t_1,t_2)),\qquad\text{ for }r=\pm L_i\pm
L_j,i\neq j.
\end{align*}
Since $\KK$ is embedded in $\KK^2$ in a obvious way, there is no
confusion if we write $x_{r}(t,0)=x_{r}(t)$ for $r=\pm 2L_i$. On the
other hand, if we write $x_{r}(a)$ where $a\in\KK^2$, then $a=(t,0)$
for some $t\in\KK$.
\begin{theorem}\label{th:5}
$SL(2n,\KK)$ $(\KK=\RR$ or $\CC)$, $n\geq 2$ is generated by
$x_{r}(a)(a\in\KK^2)$, where $0\neq r\in\Phi=\{\pm L_i\pm
L_j\}(0\leq i,j\leq n)$ subject to the relations:
\begin{align}
&x_r(a)x_r(b)=x_r (a+b),\label{for:8}\\
&[x_{r} (a), x_p (b)]=\prod_{ir+jp\in \Phi, i,j>0}
x_{ir+jp}(g_{ijpr}(a,b)), r+p\neq 0\label{for:9}\\
&[x_{r} (a), x_p (b)]=\emph{id}, \qquad 0\neq r+p\notin
\Phi\label{for:10}
\end{align}
here $a,b\in\KK^*$ and $g_{ijpr}$ are functions of $a,b$ depending
only on the structure of $SL(2n,\RR)(SL(2n,\CC))$;
\begin{align}
&h_{L_1-L_2}(t_1,0)h_{L_1-L_2}(t_2,0)=h_{L_1-L_2}(t_1\cdot
t_2,0),\label{for:11}
\end{align}
where
\begin{align*}
h_{L_1-L_2}(t,0)&=x_{L_1-L_2}(t,0)x_{L_2-L_1}(-t^{-1},0)x_{L_1-L_2}(t,0)\\
&\cdot x_{L_1-L_2}(-1,0)x_{L_2-L_1}(1,0)x_{L_1-L_2}(-1,0)
\end{align*}
for each $t\in \KK^*$.
\end{theorem}
\subsection{Proof of Theorem \ref{th:1}}
Details for Cartan action $\alpha_{0,S}$ can be found in
\cite{Damjanovic3}. $\alpha_{0,S}$ can be lifted to a $S$-action
$\tilde{\alpha}_{0,S}$ on $\widetilde{Sp}(2n,\RR)$ where
$\widetilde{Sp}(2n,\RR)$ is the universal cover of $Sp(2n,\RR)$. We
denote the new action by $\widetilde{\alpha}_{0,S}$ and the
projection from $\widetilde{Sp}(2n,\RR)$ to $Sp(2n,\RR)$ by $p$.
Following the proof-line of Theorem 2 in \cite{Damjanovic3}, we just
need to show the following 2 things: 1. Reducibility of closed
lifted cycles defined by relations (\ref{for:26}) to (\ref{for:22})
in the universal cover. 2. Trivialization of any homomorphism from
$p^{-1}(\Gamma)$ to $\RR^n$. The second one is clear by Margulis
normal subgroup theorem\cite{Margulis}. We show the proof of the
first one.

Relations of the type \eqref{for:26}, \eqref{for:27} and
\eqref{for:5} are contained in a leaf of the stable manifold for
some element of $\alpha_{0,S}$, hence they also lie in a stable leaf
 of the stable manifold for
some element of $\tilde{\alpha}_{0,S}$.

For relation \eqref{for:6} follow exactly the same way as in
Milnor¡¯s proof in [\cite{Milnor}, Theorem A1] or in
\cite{Damjanovic2}, combined with (\ref{for:4}) in proof of Lemma
\ref{le:5}, we can show that they are contractible and after an
allowable substitution, it is reducible.

For relation \eqref{for:22}, notice
$h_{2L_{n}}(-1)=$diag($1,\dots,\underset{n}-1,1,\dots,\underset{2n}-1$),
thus homotopy classes of
$\bigl(h_{2L_{n}}(-1)h_{2L_{n}}(-1)\bigl)^{k}(k\in\ZZ)$ generate the
fundermental group of $Sp(2n,\RR)$ which is isomorphic to $\ZZ$.
Hence we don't need to consider this relation in
$\widetilde{Sp}(2n,\RR)$.

Hence we finished the proof.
\section{Local differential rigidity of non-generic actions}
\subsection{Non-generic
Cartan actions on $SL(2n,\RR)/\Gamma$ and $SL(2n,\CC)/\Gamma$} We
consider Lie groups $G=SL(2n,\KK)$, $\KK=\RR$ or $\CC$, $n\geq 2$.
Its Lie algebra is the set of traceless matrices. Let
\begin{align*}
D_+=\exp\mathbb{D}_+=&\{\diag\bigl(\exp t_1,\dots,\exp t_n,\exp
(-t_1),\dots,\exp (-t_n)\bigl):\\
&(t_1,\dots,t_n)\in\RR^n\}.
\end{align*}Let $\alpha_0$ be left translations of $D_+$ on $G/\Gamma$. Let
$\Phi$ be the root system with respect to $D_+$. The roots are $\pm
L_i \pm L_j(i<j\leq n)$ with dimensions 2 and $\pm 2L_i(1\leq i\leq
n)$ with dimension 1. The set of positive roots $\Phi^{+}$ and the
corresponding set of
 simple roots $\Delta$ are
\begin{align*}
&\Phi^{+}=\{L_i-L_j\}_{i<j}\cup\{L_i+L_j\}_{i<j}\cup\{2L_i\}_{i},\\
&\Delta=\{L_i-L_{i+1}\}_i\cup\{2L_n\}.
 \end{align*}
 For $1\leq i\neq j\leq n$ the hyperplanes in $\mathbb{D_+}$ defined
by $$\mathbb{H}_{i-j} = \{(t_1, \dots , t_{n}) \in \mathbb{D_+} :
t_i = t_j\},$$
$$\mathbb{H}_{i+j} = \{(t_1, \dots , t_{n}) \in
\mathbb{D_+} : t_i+t_j=0\}\quad \text{and}$$   $$\mathbb{H}_{i} =
\{(t_1, \dots, t_{n}) \in \mathbb{D_+} : t_i=0\}$$ are
\emph{Lyapunov hyperplanes} for the action $\alpha_0$, i.e. kernels
of Lyapunov exponents of $\alpha_0$. Elements of
$\mathbb{D_+}\backslash \bigcup \mathbb{H}_{r}$(where $r=i\pm j,i$)
are \emph{regular} elements of the action. Connected components of
the set of regular elements are $Weyl$ $chambers$.

The smallest non-trivial intersections of stable foliations of
various elements of the action $\alpha_0$ are $Lyapunov$
$foliations$.

The corresponding root spaces are
\begin{align*}
 &\mathfrak{g}_{L_i+L_j}=(\KK e_{i,j+n}+\KK e_{j,i+n})_{i<j}, \qquad
 \mathfrak{g}_{L_i-L_j}=(\KK e_{i,j}+\KK e_{j+n,i+n})_{i\neq j},\\
 &\mathfrak{g}_{-L_i-L_j}=(\KK e_{j+n,i}+\KK e_{i+n,j})_{i<j},\qquad
 \mathfrak{g}_{2L_i}=\KK
 e_{i,i+n},\\
 &\mathfrak{g}_{-2L_i}=\KK e_{i+n,n}.
\end{align*}
For $t_1,t_2\in\KK$, let
\begin{align*}
&f_{L_i+L_j}(t_1,t_2)=(t_1e_{i,j+n}+t_2te_{j,i+n})_{i<j},\qquad f_{-L_i-L_j}(t_1,t_2)=(t_1e_{j+n,i}+t_2e_{i+n,j})_{i<j},\\
&f_{L_i-L_j}(t_1,t_2)=(t_1e_{i,j}+t_2e_{j+n,i+n})_{i\neq j},\qquad f_{-2L_i}(t_1)=t_1e_{i+n,n},\\
&f_{2L_i}(t_1)=t_1e_{i,i+n}.
\end{align*}
For $(t_1,t_2)\in\KK^2$, $t\in\KK$ we write
\begin{align*}
&x_\rho(t)=\exp (tf_{\rho})\qquad\text{ for }\rho=\pm 2L_i,\\
&x_r(t_1,t_2)=\exp (f_r(t_1,t_2)),\qquad\text{ for }r=\pm L_i\pm
L_j.
\end{align*}
We define foliations $F_{r}$ for $r=\pm L_i\pm L_j(i\neq j),$ and
$F_{\rho}$ for $\rho=\pm 2L_i$ for which the leaf through $x$
\begin{align}\label{for:2}
&F_{r}(x)=x_r(t_1,t_2)x,\qquad F_{\rho}(x)=x_\rho(t)x
\end{align}
consists of all left multiples of $x$   by matrices of the form
$x_r(t_1,t_2)$ or $x_\rho(t)$.

The foliations $F_r$ and $F_{\rho}$ are invariant under $\alpha_0$.
In fact, let $\mathfrak{t}=(t_1,t_2,\dots,t_n)\in\mathbb{D_+},$ for
$\forall a_1,a_2\in\KK$ we have Lie bracket relations
$$[\mathfrak{t}, f_r(a_1,a_2)]=r(\mathfrak{t})f_r(a_1,a_2),\qquad [\mathfrak{t},f_\rho(a_1)]=\rho(\mathfrak{t})f_\rho(a_1)$$
where $r(\mathfrak{t})=\pm t_i\pm t_j$ if $r=\pm L_i\pm L_j(i\neq
j)$; $\rho(\mathfrak{t})=\pm 2t_i$ if $\rho=\pm 2L_i$.

Using the basic identity for any square matrices $X,Y$:
$$\exp X\exp
Y=\exp(e^sY)\exp X, \text{ if } [X,Y]=sY,$$ it follows
\begin{align}
&\alpha_0(\mathfrak{t})\exp (f_{r}(a_1,a_2))x=\exp (e^{r(\mathfrak{t})}f_{r}(a_1,a_2))\alpha_0(\mathfrak{t})x,\\
&\alpha_0(\mathfrak{t})\exp (f_{\rho}(a_1))x=\exp
(e^{\rho(\mathfrak{t})}f_{\rho}(a_1))\alpha_0(\mathfrak{t})x.
\end{align}
Hence the leaf $F_r(x)$ is mapped into
$F_r(\alpha_0(\mathfrak{t})x)$ and $F_\rho(x)$ is mapped into
$F_\rho(\alpha_0(\mathfrak{t})x)$. Consequently the foliation $F_r$
and $F_\rho$ are contracted (corr. expanded or neutral) under
$\mathfrak{t}$ if $r(\mathfrak{t})<0$ (corr. $r(\mathfrak{t})>0$ or
$r(\mathfrak{t})=0$). If the foliation $F_r$ and $F_\rho$ are
neutral under $\alpha_0(\mathfrak{t})$, it is in fact isometric
under $\alpha_0(\mathfrak{t})$. The leaves of the orbit foliation is
$\mathcal{O}(x) = \{\alpha_0(\mathfrak{t})x : \mathfrak{t}\in
\mathbb{D_+}\}$.

The tangent vectors to the leaves in \eqref{for:2} for various $r$
and $\rho$ together with their length one Lie   brackets form a
basis of the tangent space at every $x\in X$.

Let $\mathbb{S}\subset \mathbb{D_+}$ be a closed subgroup which
contains a lattice $\mathbb{L}$ in a plane in general position and
let $S=\exp\mathbb{S}$. One can naturally think of $S$ as the image
of an injective homomorphism $i_0 : \ZZ^k\times\RR^{\ell}\rightarrow
D_+$ (where $k+\ell\geq 2)$.

The  action $\alpha_{0,S}$ of $S$ by left translations on $G/\Gamma$
is given by
\begin{align}\label{for:7}
\alpha_{0,S}(a,x)=i_0(a)\cdot x, \qquad x\in G/\Gamma.
\end{align}
If $\mathbb{P}$ is a generic $2$-plane with respect to $\Phi$ then
the foliations $F_r$ and $F_\rho$ are also Lyapunov foliations for
$\alpha_{0,\mathbb{P}}$. The leaves  of  $F_r$ and $F_\rho$ are
intersections of the leaves of stable manifolds of the action by
different elements of $\mathbb{P}$. The same holds for the action by
any regular lattice in $\mathbb{P}$ and thus for any generic
restriction $\alpha_{0,S}$ with respect to $\Phi$.

If $\KK=\RR$, the neutral foliation for a generic restriction
$\alpha_{0,S}$ is given by
$$\mathcal{N}_0(x)=\{D\cdot x: x\in SL(2n,\RR)/\Gamma\}$$
where $D$ is the set of diagonal matrices in $SL(2n,\RR)$ with
positive entries; if $\KK=\CC$, the neutral foliation is given by
$$\mathcal{N}_0(x)=\{DT\cdot x: x\in SL(2n,\CC)/\Gamma\}$$
where $T$ is the set of diagonal matrices in $SL(2n,\CC)$ whose
entries are of absolute value $1$. Thus $T$ is isomorphic to
$\mathbb{T}^{2n-1}$. Let $D_G=D$ if $G=SL(2n,\RR)$; $D_G=DT$ if
$G=SL(2n,\CC)$.

\subsection{Preliminaries of cocycles} Let $\alpha:A\times
M\rightarrow M$ be an action of a topological group $A$ on a compact
Riemannian manifold M by diffeomorphisms. For a topological group
$Y$ a $Y$-valued {\em cocycle} (or {\em an one-cocycle}) over
$\alpha$ is a continuous function $\beta : A\times M\rightarrow Y$
satisfying:
\begin{align}
\beta(ab, x) = \beta(a, \alpha(b, x))\beta(b, x)
\end{align}
 for any $a, b \in A$. A cocycle is
{\em cohomologous to a constant cocycle} (cocycle not depending on
$x$) if there exists a homomorphism $s : A\rightarrow Y$ and a
continuous transfer map $H : M\rightarrow Y$ such that for all $a
\in A $
\begin{align}\label{for:13}
 \beta(a, x) = H(\alpha(a, x))s(a)H(x)^{-1}
\end{align}
In particular, a cocycle is a {\em coboundary} if it is cohomologous
to the trivial cocycle $\pi(a) = id_Y$, $a \in A$, i.e. if for all
$a \in A$ the following equation holds:
\begin{align}
 \beta(a, x) = H(\alpha(a, x))H(x)^{-1}.
\end{align}
For more detailed information on cocycles adapted to the present
setting see \cite{Damjanovic2}.
\subsection{Paths and cycles for a collection of foliations}
In this section we recall some notation and results from
\cite{Damjanovic2}. Let $\mathcal{T}_1, \dots, \mathcal{T}_r$ be a
collection of mutually transversal continuous foliations on a
compact manifold $M$ with smooth simply connected leaves.

For $N\in\NN$ and $j_k\in\{1, \dots , r\}, k\in\{1, \dots ,N-1\}$ an
ordered set of points $p(j_1,\dots , j_{N-1}) : x_1, \dots , x_N \in
M$ is called an $\mathcal{T}$-path of length $N$ if for every $k\in
\{1, \dots ,N-1\}, x_{i+1}\in \mathcal{T}_{j_k}(x_k)$. A closed
$\mathcal{T}$-path(i.e., when $x_N=x_1$) is a $\mathcal{T}$-cycle.

A $\mathcal{T}$-cycle $p(j_1, \dots , j_{N-1}) : x_1, \dots ,
x_N=x_1 \in M$ is called $stable$ for the $A$ action $\alpha$ if
there exists a regular element $a\in A$  such that the whole cycle
$p$ is contained in a leaf of the stable foliations for the map
$\alpha(a,\cdot)$, i.e., if
\begin{align*}
\bigcap_{k=1}^N\{a:\chi_{j_k}(a)<0\}\neq \phi.
\end{align*}
\begin{definition}\label{def:1}
Let $p(j_1, \dots , j_{N-1}) : x_1, \dots , x_N$ and $p_n(j_1,
\dots, j_{N-1}) : x^n_1, \dots , x^n_N$ be two $\mathcal{T}$-paths.
Then $p=\lim_{n\rightarrow\infty}p_n$ if for all $k\in
\{1,\dots,N\}$
\begin{align*}
x_k=\lim_{n\rightarrow\infty}x^n_k.
\end{align*}
Limits of $\mathcal{T}$-cycles are defined similarly.

Two $\mathcal{T}$-cycles, $p(j_1, \dots , j_{N+1}) : x_1, \dots
,x_k,y,x_k,\dots, x_{N}=x_1$ and $p(j_1, \dots , j_{N-1}) :
x_1,\dots , x_k, x_{k+1},\dots x_N$ are said to be conjugate if
$y\in \mathcal{T}_i(x_k)$ for some $i\in \{1,\dots,r\}$. For
$\mathcal{T}$-cycles, $p(j_1, \dots , j_{N-1}) : x_1, \dots
,x_{N}=x_1$ and $p'(j_1', \dots , j_{K-1}') :x_1= x_1', \dots,
x_{K}'=x_1$ define their composition or concatenation $p\ast p'$ by
$$p\ast p'(j_1,\dots, j_{N-1}, j_1', \dots , j_{K-1}') : x_1, \dots x_N,
x_1', \dots , x_K'= x_1.$$ Let
$\mathcal{A}\mathcal{S}^s_{\mathcal{T}}(\alpha)$ denote the
collection of stable $\mathcal{T}$-cycles. Let
$\mathcal{A}\mathcal{S}_{\mathcal{T}}(\alpha)$ denote the collection
of $\mathcal{T}$-cycles which contains
$\mathcal{A}\mathcal{S}^s_{\mathcal{T}}(\alpha)$ and is closed under
conjugation, concatenation of cycles, and under the limitation
procedure defined above.
$\mathcal{A}\mathcal{S}_{\mathcal{T}}^x(\alpha)$ denotes the subset
of $\mathcal{A}\mathcal{S}_{\mathcal{T}}(\alpha)$ which contain
point $x$.

A path $p : x_1,\dots ,x_k,\dots, x_N$ reduces to a path $p^{'}:
x_1, x^{'}_2,\dots ,x^{'}_k,\dots, x_N^{'}$ via an
$\alpha$-$allowable$ $\mathcal{T}$-substitution if the
$\mathcal{T}$-cycle
\begin{align*}
p\ast p^{'}: x_1, \dots ,x_k,\dots, x_{N-1},x_N,
x_{N-1}^{'},\dots,x^{'}_2,x_1
\end{align*}
obtained by concatenation of $p$ and $p^{'}$ is in the collection
$\mathcal{A}\mathcal{S}_{\mathcal{T}}(\alpha)$.

Two $\mathcal{T}$-cycle $c_1$ and $c_2$ are $\alpha$-equivalent if
$c_1$ reduces to $c_2$ via a finite sequence of $\alpha$-allowable
$\mathcal{T}$ -substitutions. A $\mathcal{T}$-cycle we call
$\alpha$-reducible if it is in
$\mathcal{A}\mathcal{S}_{\mathcal{T}}(\alpha)$.
\end{definition}

\begin{definition}
For $N\in\NN$ and $j_k\in\{1, \dots , r\}, k\in\{1, \dots ,N\}$ an
ordered set of points $p(j_1, \dots, j_{N}) : x_1, \dots , x_N,
x_{N+1}=x_1 \in M$ is called an $\mathcal{T}$-cycle of length $N$ if
for every $k\in \{1, \dots ,N\}, x_{i+1}\in \mathcal{T}_{j_k}(x_k)$.
A $\mathcal{T}$ cycle which consists of a single point is a trivial
$\mathcal{T}$-cycle.
\end{definition}
\begin{definition}
Foliations $\mathcal{T}_1, \dots, \mathcal{T}_r$ are locally
transitive if there exists $N\in\NN$ such that for any $\varepsilon
> 0$ there exists $\delta> 0$ such that for every $x\in M$ and for every
$y\in  B_X(x, \delta)$ (where $B_M(x, \delta)$ is a $\delta$ ball in
$M$) there is a $\mathcal{T}$-path $p(j_1, \dots , j_{N-1}) : x =
x_1, x_2, \dots, x_{N-1}, x_N = y$ in the ball $B_M(x, \varepsilon)$
such that $x_{k+1}\in \mathcal{T}_{j_k}(x_k)$ and
$d_{\mathcal{T}_{j_k}(x_k)}(x_{k+1}, x_k) < 2\varepsilon$.
\end{definition}
In other words, any two sufficiently close points can be connected
by a $\mathcal{T}$-path of not more than $N$ pieces of a given
bounded length. Here, for a submanifold $Y$ in $M$, $d_Y (x, y)$
denotes the infimum of lengths of smooth curves in $Y$ connecting
$x$ and $y$.
\subsection{Cocycle rigidity for $\alpha_{0,S}$}The purpose of this section is to describe a geometric method for
proving cocycle rigidity for this action following
\cite{Damjanovic1, Damjanovic2}.
\begin{proposition}\label{po:2} Any small
$D_G$-valued H\"older cocycle over $\alpha_{0,S}$ on $G/\Gamma$ is
cohomologous to a constant cocycle via a  H\"older transfer
function.

Any $D_G$-valued $C^\infty$ small cocycle over the generic
restriction of the split Cartan action on $G/\Gamma$ is cohomologous
to a constant cocycle via a $C^\infty$ transfer function.
\end{proposition}
For a cocycle $\beta : \mathbb{S}\times G/\Gamma\rightarrow DT$ over
$\alpha_0$, we define $DT$-valued potential of $\beta$ as
\begin{align*}\label{for:1}
&\left\{\begin{aligned} &P^\gamma_a(y,x)=\lim_{n\rightarrow
+\infty}\beta(na,
y)^{-1}\beta(na,x),\qquad \gamma(a)<0\\
&P^\gamma_a(y,x)=\lim_{n\rightarrow -\infty}\beta(na,
y)^{-1}\beta(na,x),\qquad \gamma(a)>0
 \end{aligned}
 \right.
\end{align*}
where $\gamma\in\Phi$ and $a\in\mathbb{S}$. Now for any $F$-cycle
$\mathfrak{c}: x_1, \dots, x_{N+1} = x_1$ on $M$, we can define the
corresponding periodic cycle functional:
\begin{align}
\text{PCF}(\mathfrak{c})(\beta)=\prod_{i=1}^{N}P^{\gamma_i}_a(x_i,x_{i+1})(\beta)
\end{align}
where $\gamma_i\in\Phi$.

Two essential properties of the PCF which are crucial for our
purpose are that PCF is continuous and that it is invariant under
the operation of moving cycles around by elements of the action
$\alpha_{0,S}$. We now state an important proposition which is the
base of our further proof.
\begin{proposition}\emph{(Proposition 4. \cite{Damjanovic1})}
\label{po:1} Let $\alpha$ be an $\RR^k$ action by diffeomorphisms on
a compact Riemannian manifold $M$ such that a dense set of elements
of $\RR^k$ acts normally hyperbolically with respect to an invariant
foliation. If the foliations $\mathcal{F}_1,\dots,\mathcal{F}_r$ are
locally transitive and if $\beta$ is a H\"older cocycle over the
action $\alpha$ such that $F(\mathcal{C})(\beta) = 0$ for any cycle
$\mathcal{C}$ then: $\beta$ is cohomologous to a constant cocycle
via a continuous map $h:M\rightarrow Y$.
\end{proposition}
\subsection{Proof of Proposition \ref{po:2}}
At first we show the cocycle rigidity for H\"older cocycles. The
invariant foliations that we considered are $F_{r}$ and $F_\rho$
where $r=\pm L_i\pm L_j(i\neq j),$ and $\rho=\pm 2L_i$. Notice that
those foliations are smooth and their Lie brackets at length one
generate the whole tangent space. This implies that this system of
foliations is locally $1/2-H\ddot{o}lder$ transitive \cite[Section
4, Proposition 1]{Kononenko}. Every such cycle represents a relation
in the group. The word represented by this cycle can be written as a
product of conjugates of basic relations in Theorem \ref{th:5}.

Relations of the type \eqref{for:8}-\eqref{for:10} are contained in
a leaf of the stable manifold for some element of $\alpha_{0,S}$.

For relation \eqref{for:11}, if $\KK=\RR$, if doubled, follow
exactly the same way as in Milnor¡¯s proof in [\cite{Milnor},
Theorem A1] or in \cite{Damjanovic2}, we can show that they are
contractible and reducible; if $\KK=\CC$, then they are contractible
and reducible. Hence we finished the proof. Finally, to cancel
conjugations one notices that canceling
$F_r(t_1,t_2)F_r(t_1,t_2)^{-1}=\textrm{id}$ or
$F_\rho(t_1)F_\rho(t_1)^{-1}=\textrm{id}$ are also an allowed
substitution and each conjugation can be canceled inductively using
that.

Thus, the value of the periodic cycle functional for any H\"older
cocycle $\beta$ depends only on the element of $\Gamma$ this cycle
represents. Furthermore, these values provide a homomorphism $p$
from $\Gamma$ to $D_G$. The restriction of $p$ on $D$ is trivial by
Margulis normal subgroup theorem\cite{Margulis}. Notice $T$ is
abelian, thus order of $p(\Gamma)$ is bonded by
$[\Gamma:[\Gamma,\Gamma]]$ which is finite\cite[4'
Theorem]{Margulis}. By smallness of the cocycle, restriction of $p$
on $T$ vanishes.

 Hence all periodic cycle
functionals vanish on $\beta$. Now Proposition \ref{po:1} implies
that $\beta$ is cohomologous to a constant cocycle via a
H$\ddot{o}$lder transfer function.

Now consider the case of $C^{\infty}$ cocycles. Notice that  the
transfer function $H$ constructed using periodic cycle functionals
is $C^\infty$ along the  stable foliations  of  various elements of
the action. Now a general result stating that in case the smooth
distributions along with their Lie brackets generate the tangent
space at any point of a manifold a function smooth along
corresponding foliations is necessarily smooth (see \cite{Spatzier2}
for a detailed discussion and references to proofs), implies that
the transfer map $H$ is $C^\infty$.
\section{Proof of Theorems \ref{th:5}}\label{sec:8}
 The neutral foliation for a
generic restriction $\alpha_{0,S}$ is a smooth foliation, we may use
the Hirsch-Pugh-Shub structural stability theorem [\cite{shub},
Chapter 6]. Namely if $\widetilde{\alpha}_S$ is a sufficiently
$C^1$-small perturbation of $\alpha_{0,S}$ then for all elements
$a\in A$ which are regular for $\alpha_{0,S}$ and sufficiently away
from non-regular ones (denote this set by $\overline{A}$) are also
regular for $\widetilde{\alpha}_S$. The central distribution is the
same for any $a\in \overline{A}$ and is uniquely integrable to an
$\widetilde{\alpha}_S(a,\cdot)$-invariant foliation which we denote
by $\mathcal{N}$. Moreover, there is a H\"older homeomorphism
$\widetilde{h}$ of $G/\Gamma$, $C^0$ close to the $id_X$, which maps
leaves of $\mathcal{N}_0$ to leaves of $\mathcal{N}$:
$\widetilde{h}\mathcal{N}_0 = \mathcal{N}$. This homeomorphism is
uniquely defined in the transverse direction, i.e. up to a
homeomorphism preserving $\mathcal{N}$. Furthermore, $\widetilde{h}$
can be chosen smooth and $C^1$ close to the identity along the
leaves of $\mathcal{N}_0$ although we will not use the latter fact.
Clearly the leaves of the foliation $\mathcal{N}_0$ are preserved by
every $a\in \overline{A}$. The action $\alpha_S$ is H\"older but it
is smooth and $C^1$-close to $\alpha_{0,S}$ along the leaves of the
neutral foliation $\mathcal{N}_0$.

 Let us
define an action $\alpha_S$ of $S$ on $G/\Gamma$ as the conjugate of
$\widetilde{\alpha}_S$ by the map $\widetilde{h}$ obtained from the
Hirsch-Pugh- Shub stability theorem: $$\alpha_S :=
\widetilde{h}^{-1}\circ \widetilde{\alpha}_S\circ \widetilde{h}$$
Since the action  $\alpha_S$ is a $C^0$ small perturbation of
$\alpha_{0,S}$ along the leaves of the neutral foliation of
$\alpha_{0,S}$ whose leaves are $\{D_G\cdot x : x\in X\}$, we have
that $\alpha_S$ is given by a map $\beta: (\ZZ^k\times
\RR^\ell)\times G/\Gamma\rightarrow D_G$ by
\begin{align}\label{for:12}
\alpha_S(a, x) = \beta(a, x) \cdot \alpha_{0,S}(a, x)
\end{align}
 for $a\in \ZZ^k\times \RR^\ell$ and $x\in G/\Gamma$.
Notice that since $\alpha_S$ is a small perturbation of the action
by left translations $\alpha_{0,S}$, it can be lifted to a
$S$-action $\overline{\alpha}_S$ on $G$ commuting with the right
$\Gamma$ action on $G$, and $\beta$ is lifted to a cocycle
$\overline{\beta}$ over $\overline{\alpha}_S$ (for more details see
[\cite{ Margulis2}, example 2.3]). In particular we have:
$$\overline{\beta}(ab,x)=\overline{\beta}(a,\overline{\alpha}_S(b,x))\overline{\beta}(b,x).$$

It follows that since $\alpha_S$ is H\"{o}lder,
$\overline{\beta}(a,x)$ is small H\"{o}lder cocycle over the action
$\overline{\alpha}_S$, due to the smallness of the perturbation.

Let $U:U_1,\dots,U_r$ denote the invariant unipotent foliations for
the lifted action $\overline{\alpha}_{0,S}$ of $\alpha_{0,S}$ on $G$
which projects to invariant Lyapunov foliations
 for $\alpha_{0,S}$;
and let $T:T_1,\dots,T_r$ denote invariant Lyapunov foliations for
lifted  $\overline{\alpha}_S$ which projects to invariant Lyapunov
foliations for $\alpha_S$. Notice that the latter foliations  have
only H\"older leaves but we are justified in calling them Lyapunov
foliations since they are images of  Lyapunov foliations  for a
smooth perturbed action under a H\"older conjugacy.
 Denote  the neutral foliation $\mathcal{N}_0$ on $G$ by $N_0$. An immediate corollary of the result of
Brin and Pesin \cite{pesin} on persistence of local transitivity of
stable and unstable foliations of a partially hyperbolic
diffeomorphisms and the fact that the collection of homogeneous
Lyapunov foliations $U:U_1,\dots,U_r$ is locally transitive and
$T:T_1,\dots,T_r$ is transitive and they are leafwise $C^0$ close.
Following the proof line closely with only trivial modifications
from those of [Section 6.2, 6.2 and 6.4 \cite{Damjanovic2}], and
[Section 5.3,5.4, \cite{Damjanovic3}], we can show $U$-cycles and
$T$-cycles project to each other along the neutral foliations
(precise definitions are in [Section 6.2,\cite{Damjanovic2}]), which
implies:
\begin{proposition}\label{po:3}
The lifted cocycle for the perturbed action $\overline{\alpha}_S$ is
cohomologous to a constant cocycle.
\end{proposition}
The value of the periodic cycle functional for $H\tilde{o}lder$
cocycle $\beta$ over $\widetilde{\alpha}_S$ or its $H\tilde{o}lder$
conjugate $\alpha_S$
 depends only on the element of $\Gamma$ this cycle
represents. Using the same trick as in proof of Proposition
\ref{po:2}, we can show every homomorphism from $\Gamma$ to $D_G$ is
trivial.

$\beta$ is cohomologous to a small constant cocycle $g:\ZZ^k\times
\RR^\ell\rightarrow D_G$ via a continuous transfer map
$H:G/\Gamma\rightarrow D_G$ which can be chosen close to identity in
$C^0$ topology if the perturbation $\widetilde{\alpha}_S$ small in
$C^2$ topology.

Let us consider the map $h'(x):= H^{-1}(x)\cdot x$. We have from the
cocycle equation \eqref{for:12} and the cohomology equation
\eqref{for:13}
$$h'({\alpha}_S(a, x)) = \alpha_{0,\widetilde{S}}(a, h'(x))$$ where $\alpha_{0,\widetilde{S}}(a, x) := i(a)\cdot x$,
where $i(a) := g(a)i_0(a), a\in A$ and $i_0$ is as in \eqref{for:7}.
Since the map $h'$ is $C^0$ close to the identity it is surjective
and thus the action $\alpha_S$ is semi-conjugate to the standard
perturbation $\alpha_{0,\widetilde{S}}$ of $\alpha_{0,S}$, i.e.
$\alpha_{0,\widetilde{S}}$ is a factor of $\alpha_S$. It is enough
to prove that $h'$ is injective. By simple transitivity of
$U$-holonomy group and the fact that there is no non-trivial element
in $DT$ such that all its powers are small [Section 7.1
\cite{Damjanovic2}] we have:
\begin{proposition}\emph{(Section 6.1 \cite{Damjanovic2})}
The map $h'$ is a homeomorphism and hence provides a topological
conjugacy between $\alpha_S$ and $\alpha_{0,\widetilde{S}}$.
\end{proposition}
Now by letting $h:=h'\widetilde{h}^{-1}$ we have
$$h\circ \widetilde{\alpha}_Sh^{-1}=\alpha_{0,\widetilde{S}}$$
thus there is a topological conjugacy between $\widetilde{\alpha}_S$
and $\alpha_{0,\widetilde{S}}$. The smoothness of this homeomorphism
follows as in \cite{Spatzier}, \cite{Damjanovic2} or
\cite{Margulis2}, by the general Katok-Spatzier theory of
non-stationary normal forms for partially hyperbolic abelian
actions.
\section{Proof of Theorem \ref{th:3}}\label{sec:6}
\subsection{Basic settings in $Sp(2n,\RR)$} Let $Q$ be a non-degenerate standard skew-symmetric
bilinear form on $\RR^{2n}$. Take $Q$ to be the bilinear form given,
in terms of a basis $e_1,\dots,e_{2n}$ for $\RR^{2n}$, by
$Q(e_i,e_{i+n})=1$, $Q(e_{i+n},e_{i})=-1$ and $Q(e_i,e_j)=0$
otherwise.

Using  this base,  the  Lie algebra $sp(2n,\RR)$ of $Sp(2n,\RR)$ can
be represented  as $2n\times 2n$ matrices
$$ \begin{pmatrix}A_1 & A_2 \\
A_3 & A_4 \\
 \end{pmatrix},$$
  where $A_1, A_2, A_3, A_4$ are $n\times n$ matrices satisfying $A_1=-A^{\tau}_4$
and $A_2$ and $A_3$ are symmetric.

We denote by $S$ the set of $2n\times 2n$ diagonal matrices in
$Sp(2n,\RR)$ with positive entries. Let $\Phi$ be the root system
with respect to $S$. The roots are $\pm L_i \pm L_j(i<j\leq n) $ and
$\pm 2L_i(1\leq i\leq n)$. The set of positive roots $\Phi^{+}$ and
the corresponding set of
 simple roots $\Delta$ are
 \begin{align*}
&\Phi^{+}=\{L_i-L_j\}_{i<j}\cup\{L_i+L_j\}_{i<j}\cup\{2L_i\}_{i},\\
&\Delta=\{L_i-L_{i+1}\}_i\cup\{2L_n\}.
 \end{align*}
Let $1\leq i,j\leq n, i\neq j$ be two distinct indices and let
$\exp$ be the exponentiation map for matrices.

The corresponding root spaces are
\begin{align*}
 &\mathfrak{g}_{L_i+L_j}=\RR(e_{i,j+n}+e_{j,i+n})_{i<j}, \qquad
 \mathfrak{g}_{L_i-L_j}=\RR(e_{i,j}-e_{j+n,i+n})_{i\neq j},\\
 &\mathfrak{g}_{-L_i-L_j}=\RR(e_{j+n,i}+e_{i+n,j})_{i<j},\qquad \mathfrak{g}_{2L_i}=\RR
 e_{i,i+n},\\
 &\mathfrak{g}_{-2L_i}=\RR e_{i+n,n}.
\end{align*}
Let
\begin{align*}
&f_{L_i+L_j}=(e_{i,j+n}+e_{j,i+n})_{i<j},\qquad
f_{L_i-L_j}=(e_{i,j}-e_{j+n,i+n})_{i\neq j},\\
&f_{-L_i-L_j}=(e_{j+n,i}+e_{i+n,j})_{i<j},\qquad
f_{2L_i}=e_{i,i+n},\\
&f_{-2L_i}=e_{i+n,n}.
\end{align*}

For $t\in \RR$ we write
\begin{align*}
&x_r(t)=\exp (tf_{r})\in U^{r}_\RR\qquad\text{ for }r\in\Phi.
\end{align*}
Let
\begin{align*}
w_{r}(t)=x_{r}(t)x_{-r}(-t^{-1})x_{r}(t),\qquad t\in \RR^*
\end{align*}
where
\begin{align*}
x_i=x_{r}(t)\forall i,\qquad y_i=x_{-r}(-t^{-1})\forall i.
\end{align*}
Correspondingly, we define
\begin{align*}
&h_r(t)=w_{r}(t)w_{r}(1)^{-1},\qquad t\in\RR^*,r\in\Phi.
\end{align*}

\subsection{Relations in universal central extension}\label{sec:4}
For $\gamma, \beta\in \Phi$ such that $\gamma\neq -\beta$, it is
known that
$$[x_\gamma, x_\beta]\subset\prod_{\chi=i\gamma+j\beta, i,j\geq 1}x_\chi.$$
This clearly gives rise to numbers $g_{ijpr}$ satisfying
\begin{align}
& x_r(a)x_r(b)=x_r (a+b)\label{for:14} \\
 &[x_{r} (a), x_p (b)]=\prod_{ir+jp\in
\Phi, i,j>0}
x_{ir+jp}(g_{ijpr}a^ib^j), r+p\neq 0,\\
&[x_{r} (a), x_p (b)]=\textrm{id}, \qquad 0\neq r+p\notin
\Phi.\label{for:15}
\end{align}
If $\widetilde{G}$ is the group defined by relations
\eqref{for:14}--\eqref{for:15}, and if $\pi_1$ is the natural
homomorphism from $\widetilde{G}$ to $Sp(2n,\RR)$, then $(\pi_1,
\widetilde{G})$ is a universal central extension of $Sp(2n,\RR)$.
(For the proof of this and other elementary properties of a
universal central extension, one may refer to [\cite{Steinberg2},
Section 7].)
 We write for
$x_\rho(t)$, the corresponding element in $\widetilde{G}$ by
$\tilde{x}_\rho(t)$. Then $\tilde{w}_\rho(u)$, $\tilde{h}_\rho(u)$,
$u\in\RR^*$ are obviously defined elements of $\widetilde{G}$.
\begin{lemma}\label{le:22}
If $a,t_1\in\RR^*$,  the following hold in $\widetilde{G}$(and hence
in $G$ too).

\begin{itemize}
\item[1]$\tilde{w}_{2L_n}(a)\tilde{w}_{L_{n-1}-L_n}(t_1)\tilde{w}_{2L_n}(a)^{-1}=\tilde{w}_{L_{n-1}+L_n}\left(-at_1\right)$,\\

\item[2]$\tilde{w}_{2L_n}(a)\tilde{w}_{L_{n-1}+L_n}(t_1)\tilde{w}_{2L_n}(a)^{-1}=\tilde{w}_{L_{n-1}-L_n}\left(a^{-1}t_1\right)$,\\

\item[3]$\tilde{w}_{L_{n-1}-L_n}(t_1)\tilde{w}_{2L_n}(a)\tilde{w}_{L_{n-1}-L_n}(t_1)^{-1}=\tilde{w}_{2L_{n-1}}(at_1^2)$,\\

\item[4]$\tilde{w}_{L_{n-1}-L_n}(t_1)\tilde{w}_{2L_{n-1}}(a)\tilde{w}_{L_{n-1}-L_n}(t_1)^{-1}=\tilde{w}_{2L_{n}}\left(at_1^{-2}\right)$.
\end{itemize}
Hence,
\begin{itemize}
\item[5]$\tilde{h}_{L_{n-1}-L_n}(t_1)\tilde{w}_{2L_n}(a)\tilde{h}_{L_{n-1}-L_n}(t_1)^{-1}=\tilde{w}_{2L_{n}}(at_1^{-2})$,\\

\item[6]$\tilde{w}_{2L_n}(a)\tilde{h}_{L_{n-1}-L_n}(t_1)\tilde{w}_{2L_n}(a)^{-1}\\
=\tilde{h}_{L_{n-1}+L_n}(-at_1)\tilde{h}_{L_{n-1}+L_n}(-a)^{-1}$.\\
\end{itemize}
\end{lemma}
\begin{proof} Formulae (1) to (4) are proved easily by computations using
\cite[p.81]{Steinberg2}. (5) and (6) are nothing but (1) and (3)
applied twice.
\end{proof}
We use diag$(a_{k_1},\cdots,a_{k_i})$ to denote the $2n\times 2n$
diangonal matrix with $k_{j}$-th diagonal element $a_{k_{j}}$ and
remaining diagonal elements $1$.
\begin{lemma}\label{le:4} For $\gamma\in\Phi$,
denote by $\tilde{H}_{\gamma}$ the subgroup generated by
$\tilde{h}_{\gamma}(t)(t\in\RR^*)$, Let $\tilde{H}$ be he subgroup
generated by $\{\tilde{H}_{\alpha}, \alpha\in\Phi\}$.
\begin{align*}
&(1) \ker(\pi_1)\cap \tilde{H}_{L_1-L_2}=\{\prod_i
\tilde{h}_{L_1-L_2}(t_i)\mid \text{ \emph{with} }\prod_i t_i=1\}.\\
&(2)  \ker(\pi_1)\cap \tilde{H}_{r}=\ker(\pi_1)\cap
\tilde{H}_{L_1-L_2}, \qquad \text{ \emph{for} }r=\pm L_i\pm
L_j(i\neq j).\\
&(3)
\ker(\pi_1)=\prod_{\gamma\in\Delta}\bigl(\ker(\pi_1)\cap\tilde{H}_{\gamma}\bigl).
\end{align*}
\end{lemma}
\begin{proof} (1). Notice
$\pi_1(\tilde{h}_{L_1-L_2}(t))=$diag$\left(t_1,(t^{-1})_2,(t^{-1})_{1+n},t_{2+n}\right)$.
Thus (1) is clear.

It follows from (1) that $\ker(\pi_1)\cap \tilde{H}_{L_1-L_2}$ is
generated by elements
$$\tilde{h}_{L_1-L_2}(t_1)\tilde{h}_{L_1-L_2}(t_2)\tilde{h}_{L_1-L_2}(t_1t_2)^{-1},\text{ where }t_1,t_2\in\RR^*.$$

(2) We can prove similarly that $\ker(\pi_1)\cap
\tilde{H}_{r}$($r=\pm L_i\pm L_j$) is generated by elements
$\tilde{h}_{r}(t_1)\tilde{h}_{r}(t_2)\tilde{h}_{r}(t_1t_2)^{-1}$.
Since these simple roots belong to the same orbit under the Weyl
group, an argument similar to one in [\cite{Moore}, Lemma 8.2] shows
that $\ker(\pi_1)\cap \tilde{H}_{r}\subseteq \ker(\pi_1)\cap
\tilde{H}_{L_1-L_2}$ for all roots $r=\pm L_i\pm L_j$. This proves
(2).

(3) By \cite[p.40]{Steinberg2}, we have $\ker(\pi_1)\subseteq
\tilde{H}$. Using se a method similar to that in the proof of
\cite[7.7]{Steinberg}, we have
$\tilde{H}=\prod_{\gamma\in\Delta}\tilde{H}_{\gamma}$. Using the
simple connectedness of $Sp(2n,\CC)$ over $\CC$\cite[p.24]{Deodhar},
we get (3).
\end{proof}

For $t_1,t_2\in\RR^{*}$, we define:
\begin{align*}
&\{t_1,t_2\}=\tilde{h}_{L_1-L_2}(t_1)\tilde{h}_{L_1-L_2}(t_2)\tilde{h}_{L_1-L_2}(t_1t_2)^{-1}.
\end{align*}
Now in exactly the same manner as the proof in the appendix of
\cite{Moore}, we prove that these $\{t_1,t_2\}$'s satisfy the
conditions
\begin{lemma}\label{le:24}
\begin{align*}
\{t_1, t_2\}&=\{t_2,t_1\}^{-1}\qquad \forall t_1,t_2\in \RR^*,\\
\{t_1, t_2\cdot t_3\}&=\{t_1,t_2\}\cdot\{t_1,t_3\}\qquad \forall t_1, t_2,t_3\in \RR^*,\\
\{t_1\cdot t_2,
t_3\}&=\{t_1,t_3\}\cdot\{t_2,t_3\}\qquad \forall t_1, t_2,t_3\in \RR^*,\\
\{t,1-t\}&=1\qquad \forall t\in \RR^*, t\neq 1,\\
\{t,-t\}&=1\qquad \forall t\in\RR^*.
\end{align*}
\end{lemma}
Let $\tilde{H}_{0}$ denote the cyclic group generated by
$\tilde{h}_{2L_n}(-1)\tilde{h}_{2L_n}(-1)$ and $\tilde{H}_{c}$
denote the cyclic group generated by $\tilde{h}_{2L_n}(-1)$. To
prove Theorem \ref{th:3}, it is equivalent to prove:
\begin{proposition}\label{th:9}
$\ker(\pi_1)=\bigl(\ker(\pi_1)\cap \tilde{H}_{L_1-L_2}\bigl)\cdot
\tilde{H}_0$.
\end{proposition}
The proof of this proposition relies on the following results.
\begin{lemma}\label{le:5}
\begin{itemize}
\item[(i)]$\tilde{H}_{2L_n}\subseteq \tilde{H}_{L_{n-1}-L_n}\cdot
\tilde{H}_{L_{n-1}+L_n}\cdot \tilde{H}_c$.

\item[(ii)]$\ker(\pi_1)\cap\tilde{H}_{2L_n}\subseteq(\ker(\pi_1)\cap\tilde{H}_{L_{1}-L_2})\cdot\tilde{H}_0$.
\end{itemize}
\end{lemma}
\begin{proof}
(i) Using Lemma \ref{le:22}, for $\forall t\in\RR^*$, let
$z^2=\abs{t}$ we have
\begin{align*}
\tilde{h}_{2L_n}(t)&=\tilde{w}_{2L_n}(t)\tilde{w}_{2L_n}(-1)\\
&=\tilde{h}_{L_{n-1}-L_n}(z^{-1})\tilde{w}_{2L_n}(tz^{-2})\tilde{h}_{L_{n-1}-L_n}(z^{-1})^{-1}\tilde{w}_{2L_n}(-1)\\
&=\tilde{h}_{L_{n-1}-L_n}(z^{-1})\tilde{w}_{2L_n}(t\abs{t}^{-1})\tilde{h}_{L_{n-1}-L_n}(z^{-1})^{-1}\tilde{w}_{2L_n}(-1).
\end{align*}
If $t>0$ we have
\begin{align*}
&\tilde{h}_{L_{n-1}-L_n}(z^{-1})\tilde{w}_{2L_n}(t\abs{t}^{-1})\tilde{h}_{L_{n-1}-L_n}(z^{-1})^{-1}\tilde{w}_{2L_n}(-1)\\
=&\tilde{h}_{L_{n-1}-L_n}(z^{-1})(\tilde{w}_{2L_n}(1)\tilde{h}_{L_{n-1}-L_n}(z^{-1})^{-1}\tilde{w}_{2L_n}(-1))\\
=&\tilde{h}_{L_{n-1}-L_n}(z^{-1})\tilde{h}_{L_{n-1}+L_n}(-
1)\tilde{h}_{L_{n-1}+L_n}(-z^{-1} )^{-1}.
\end{align*}
If $t<0$ we have
\begin{align*}
&\tilde{h}_{L_{n-1}-L_n}(z^{-1})\tilde{w}_{2L_n}(t\abs{t}^{-1})\tilde{h}_{L_{n-1}-L_n}(z^{-1})^{-1}\tilde{w}_{2L_n}(-1)\notag\\
=&\tilde{h}_{L_{n-1}-L_n}(z^{-1})\tilde{w}_{2L_n}(-1)\tilde{h}_{L_{n-1}-L_n}(z^{-1})^{-1}\tilde{w}_{2L_n}(-1)\notag\\
=&\tilde{h}_{L_{n-1}-L_n}(z^{-1})\tilde{h}_{L_{n-1}+L_n}(z^{-1}
)^{-1}\tilde{h}_{2L_n}(-1).
\end{align*}
Especially, if $t=1$, $z=-1$we have
\begin{align}
e=\tilde{h}_{L_{n-1}-L_n}(-1)\tilde{h}_{L_{n-1}+L_n}(-1).\label{for:1}
\end{align}
Especially, if $t=-1$, $z=-1$we have
\begin{align}
e=\tilde{h}_{L_{n-1}-L_n}(-1)\tilde{h}_{L_{n-1}+L_n}(-1)^{-1}.\label{for:3}
\end{align}
Thus we get
\begin{align}
e=\bigl(\tilde{h}_{L_{n-1}-L_n}(-1)\bigl)^{2}.\label{for:4}
\end{align}
 Hence we proved (i).

 (ii) By Lemma \ref{le:23} and (i), any $h\in \tilde{H}_{2L_n}$ can be written as
$$h=\tilde{h}_{L_{n-1}-L_n}(t_1)\tilde{h}_{L_{n-1}+L_n}(t_2)h_1h_2$$
where $t_1,t_2\in\RR^*$, $h_1\in\ker(\pi_1)\cap
\tilde{H}_{L_{1}-L_2}$ and $h_2\in \tilde{H}_c$.

If $\pi_1(h)=I_{2n}$, we have $t_1=t_2=\pm 1$, and
$\pi_1(h_2)=I_{2n}$. If $t_1=t_2=1$, we have $h=h_1h_2$. If
$t_1=t_2=-1$, by (\ref{for:1}) we still have $h=h_1h_2$. Notice
$\pi_1(\tilde{h}_{2n}(-1))=$diag$((-1)_n,(-1)_{2n})$, it follows
$h_2=\bigl(\tilde{h}_{2n}(-1)\bigl)^{2k}$, $k\in\ZZ$. Hence we
proved (ii).
\end{proof}
\subsection{Proof of Proposition \ref{th:9}}
By (3) of Lemma \ref{le:4}, $$\ker\pi_1=\prod_{\alpha\in
\Delta}(\ker\pi_1\bigcap \tilde{H}_{\alpha}),$$ where
$\Delta=\{L_i-L_{i+1}\}_i\cup\{2L_n\}$.

By (ii) of Lemma \ref{le:5}, we have
$$\ker\pi_1\subseteq(\ker\pi_1\bigcap \tilde{H}_{L_1-L_2})\cdot \tilde{H}_0.$$
The inverse inclusion is obvious. Hence we finished the proof.
\section{Proof of Theorem \ref{th:5}}\label{sec:3}
\subsection{Basic settings}We follow some notations in Section \ref{sec:6}. We consider Lie groups $G=SL(2n,\KK)$, $\KK=\RR$ or $\CC$, $n\geq
2$. Its Lie algebra is the set of traceless matrices. Let
\begin{align*}
D_+=&\{\diag\bigl(\exp t_1,\dots,\exp t_n,\exp
(-t_1),\dots,\exp (-t_n)\bigl):\\
&(t_1,\dots,t_n)\in\RR^n\}.
\end{align*}Let
$\Phi$ be the root system with respect to $D_+$. The roots are $\pm
L_i \pm L_j(i<j\leq n)$ with dimensions 2 and $\pm 2L_i(1\leq i\leq
n)$ with dimension 1. The set of positive roots $\Phi^{+}$ and the
corresponding set of
 simple roots $\Delta$ are
\begin{align*}
&\Phi^{+}=\{L_i-L_j\}_{i<j}\cup\{L_i+L_j\}_{i<j}\cup\{2L_i\}_{i},\\
&\Delta=\{L_i-L_{i+1}\}_i\cup\{2L_n\}.
 \end{align*}
Let $1\leq i,j\leq n, i\neq j$ be two distinct indices and let
$\exp$ be the exponentiation map for matrices.

The corresponding root spaces are
\begin{align*}
 &\mathfrak{g}_{L_i+L_j}=(\RR e_{i,j+n}+\RR e_{j,i+n})_{i<j}, \qquad
 \mathfrak{g}_{L_i-L_j}=(\RR e_{i,j}+\RR e_{j+n,i+n})_{i\neq j},\\
 &\mathfrak{g}_{-L_i-L_j}=(\RR e_{j+n,i}+\RR e_{i+n,j})_{i<j},\qquad \mathfrak{g}_{2L_i}=\RR
 e_{i,i+n},\\
 &\mathfrak{g}_{-2L_i}=\RR e_{i+n,n}.
\end{align*}
Let
\begin{align*}
f^1_{L_i+L_j}(t)&=(te_{i,j+n})_{i<j},\qquad
&f^2_{L_i+L_j}(t)&=(te_{j,i+n})_{i<j},\\
f^1_{L_i-L_j}(t)&=(te_{i,j})_{i\neq j},\qquad &f^2_{L_i-L_j}(t)&=(te_{j+n,i+n})_{i\neq j}\\
f^1_{-L_i-L_j}(t)&=(te_{j+n,i})_{i<j},\qquad
&f^2_{-L_i-L_j}(t)&=(te_{i+n,j})_{i<j}\\
f_{2L_i}(t)&=te_{i,i+n}, \qquad &f_{-2L_i}(t)&=te_{i+n,n}.
\end{align*}

For $(t_1,t_2)\in\KK^2$, $t\in\KK$ we write
\begin{align*}
&x_\rho(t)=\exp (tf_{\rho})\qquad\text{ for }\rho=\pm 2L_i,\\
&x_r(t_1,t_2)=\exp (f^1_r(t_1))\exp (f^2_r(t_2)),\qquad\text{ for
}r=\pm L_i\pm L_j.
\end{align*}

Let
\begin{align*}
&w_{\rho}(t)=x_{\rho}(t)x_{-\rho}(-t^{-1})x_{\rho}(t),\qquad t\in
\KK^*,\rho=\pm 2L_i.
\end{align*}
For $r=\pm L_i\pm L_j(i\neq j)$, $(t_1,t_2)\in\KK^*\times\KK^*$, let
\begin{align*}
&w_{r}(t_1,t_2)=x_{r}(t_1,t_2)x_{-r}(-t_1^{-1},-t_2^{-1})x_{r}(t_1,t_2);
\end{align*}
for $t\in\KK^*$ let
\begin{align*}
&w_{r}(t,0)=x_{r}(t,0)x_{-r}(-t^{-1},0)x_{r}(t,0),\\
&w_{r}(0,t)=x_{r}(0,t)x_{-r}(0,-t^{-1})x_{r}(0,t).
\end{align*}
Correspondingly, we define
\begin{align*}
&h_\rho(t)=w_{\rho}(t)w_{\rho}(1)^{-1},\qquad t\in\KK^*,\rho=\pm
2L_i,
\end{align*}
for $r=\pm L_i\pm L_j(i\neq j)$, $(t_1,t_2)\in\KK^*\times\KK^*$, let
\begin{align*}
&h_{r}(t_1,t_2)=w_{r}(t_1,t_2)w_{r}(-1,-1);
\end{align*}
for $t\in\KK^*$ let
\begin{align*}
&h_{r}(t,0)=w_{r}(t,0)w_{r}(-1,0),\\
&h_{r}(0,t)=w_{r}(0,t)w_{r}(0,-1).
\end{align*}
Let us write $p(\pi)$ the permutation matrix corresponding to the
permutation $\pi$, that is, the $i,j$ entry of $p(\pi)$ is $1$ if
$i=\pi(j)$ and zeros otherwise. With these notations we have:
\begin{align*}
w_{L_i-L_j}(t_1,t_2)=p(\pi)\text{diag}\left((-t_1^{-1})_i,(t_1)_j,(t_2)_{i+n},(-t_2^{-1})_{j+n}\right),
\quad(t_1,t_2)\in\RR^*\times\RR^*
\end{align*}
where $\pi$ only permutes $(i,j)$ and $(i+n,j+n)$ while fixes other
numbers.
\begin{align*}
w_{L_i-L_j}(t_1,0)=p(\pi)\text{diag}\left((-t_1^{-1})_i,(t_1)_j\right),
\qquad t_1\in\RR^*
\end{align*}
where $\pi$ only permutes $(i,j)$ while fixes other numbers.
\begin{align*}
w_{L_i-L_j}(0,t_2)=p(\pi)\text{diag}\left((t_2)_{i+n},(-t_2^{-1})_{j+n}\right),
\qquad t_2\in\RR^*
\end{align*}
where $\pi$ only permutes $(i+n,j+n)$ while fixes other numbers.
\begin{align*}
w_{L_i+L_j}(t_1,t_2)=p(\pi)\text{diag}\left((-t_1^{-1})_i,(-t_2^{-1})_j,(t_2)_{i+n},(t_1)_{j+n}\right),\quad
(t_1,t_2)\in\RR^*\times\RR^*
\end{align*}
where $\pi$ only permutes $(i,j+n)$ and $(j,i+n)$ while fixes other
numbers.
\begin{align*}
w_{L_i+L_j}(0,t_2)=p(\pi)\text{diag}\left((-t_2^{-1})_j,(t_2)_{i+n}\right),\quad
t_2\in\RR^*
\end{align*}
where $\pi$ only permutes $(j,i+n)$ while fixes other numbers.
\begin{align*}
w_{L_i+L_j}(t_1,0)=p(\pi)\text{diag}\left((-t_1^{-1})_i,(t_1)_{j+n}\right),\quad
t_1\in\RR^*
\end{align*}
where $\pi$ only permutes $(i,j+n)$ while fixes other numbers.
\begin{align*}
w_{2L_i}(t)=p(\pi)\text{diag}\bigl((-t^{-1})_i,t_{i+n}),\qquad
\text{ for }a\in \RR^*,
\end{align*}
where $\pi$ only permutes $(i,i+n)$ while fixes other numbers. Let
$W_0$ be the set composed of all permutations stated above. Then
$W_0$ is just $S_{2n}$, the permutation group on $2n$ elements.

The root system is not stable under $W_0$. For example, if $w\in
W_0$ permutes $i$ and $j$ only, then $w(L_i-L_\ell)$ is not a root
for any $\ell\leq n$. If we consider the restricted roots, that is
let $\gamma^\delta$, $\gamma\in\Phi$, $\delta=1,2$ be the root
restricted on root space $f^\delta_\gamma$, then restricted root
system are stable under $W_0$.

 We denote by
$x_\gamma(\gamma\in\Phi)$ the subgroup generated by
$x_\gamma(t),t\in\KK$ or $t\in\KK^2$. We can construct the extension
as was done in Section \ref{sec:4} with respect to $\Phi$. We still
get $\widetilde{G}$ and a well-defined homomorphism
$\pi_1:\widetilde{G}\rightarrow SL(2n,\KK)$. We can also define
elements $\tilde{x}_{r}(t_1,t_2)$, $\tilde{x}_{\rho}(t)$,
$\tilde{w}_{r}(t_1,t_2)$, $\tilde{w}_{\rho}(t)$,
$\tilde{h}_r(t_1,t_2)$, $\tilde{h}_\rho(t)$, $\tilde{x}_\gamma$ etc.
as was done in Section \ref{sec:4}.
\begin{remark}
Notice now we don't know if ($\widetilde{G}$, $\pi_1$) is central or
not, not to mention universal central or not(in fact, we can prove
it is). But $\pi_1$ is surjective since $x_\gamma$ and their Lie
brackets generate the the whole group.
\end{remark}
It is clear certain relations hold both in $\widetilde{G}$ and
$SL(2n,\KK)$. To simplify notation, we write for $f\in
x_\gamma(\gamma\in\Phi)$ the corresponding element in
$\widetilde{G}$ by $\tilde{x}_\gamma(f)$. The notation coincides
with the former one. We have
\begin{lemma}\label{le:2}
If $\gamma, \beta=\pm L_i\pm L_j(i\neq j)$,
$(u_1,u_2)\in\KK^2\backslash 0$, $u,v_1,v_2\in\KK^*$ then
\begin{itemize}
\item[1]$\tilde{w}_{\gamma}(u_1,u_2)\tilde{x}^\delta_{\beta}(v)\tilde{w}_{\gamma}(u)^{-1}\\=
\tilde{x}_{w_{\gamma}(\beta^\delta)}\bigl(w_{\gamma}(u_1,u_2)x^\delta_{\beta}(v)w_{\gamma}(u_1,u_2)^{-1}\bigl)$.
\item[2]$\tilde{w}_{2L_i}(u)\tilde{x}^\delta_{\beta}(v)\tilde{w}_{2L_i}(u)^{-1}\\=
\tilde{x}_{w_{2L_i}(\beta^\delta)}\bigl(w_{2L_i}(u)x^\delta_{\beta}(v)w_{2L_i}(u)^{-1}\bigl)$.
\end{itemize}
\end{lemma}

\begin{proof}
It is easily proved by computations using \cite[p.40]{Steinberg2}
and 1.10--1.12 in \cite{Deodhar}.
\end{proof}
\begin{lemma}\label{le:1}
$\tilde{w}_\gamma(t_1,t_2)(t_1,t_2\in\KK^*)$ with $\gamma=\pm L_i\pm
L_j(i\neq j)$ are generated by $\tilde{w}_\beta(t,0)$,
$\tilde{w}_\beta(0,t)$ and $\tilde{w}_{2L_i}(t)$ where
$\beta=L_i-L_j(i<j)$, $t\in\KK^*$.
\end{lemma}
\begin{proof} For $a\in\KK^*$, $t_1,t_2\in\KK^*$, keep using Lemma \ref{le:2} we have
\begin{align*}
&\tilde{w}_{L_i+L_j}(a,0)\tilde{w}_{L_i-L_j}(t_1,t_2)\tilde{w}_{L_i+L_j}(a,0)^{-1}\\
&=\tilde{w}_{L_i+L_j}(a,0)\tilde{x}_{L_i-L_j}(t_1,t_2)\tilde{x}_{L_j-L_i}(-t_1^{-1},-t_2^{-1})
\tilde{x}_{L_i-L_j}(t_1,t_2)\tilde{w}_{L_i+L_j}(a,0)^{-1}\\
&=\tilde{x}_{-2L_j}(-a^{-1}t_1)\tilde{x}_{2L_i}(at_2)\tilde{x}_{2L_j}(at_1^{-1})
\tilde{x}_{-2L_i}(-a^{-1}t_2^{-1})\tilde{x}_{-2L_j}(-a^{-1}t_1)\tilde{x}_{2L_i}(at_2)\\
&=\tilde{x}_{-2L_j}(-a^{-1}t_1)\tilde{x}_{2L_j}(at_1^{-1})\tilde{x}_{2L_i}(at_2)
\tilde{x}_{-2L_i}(-a^{-1}t_2^{-1})\tilde{x}_{2L_i}(at_2)\tilde{x}_{-2L_j}(-a^{-1}t_1)\\
&=\tilde{x}_{-2L_j}(-a^{-1}t_1)\tilde{x}_{2L_j}(at_1^{-1})\tilde{w}_{2L_i}(at_2)
\tilde{x}_{-2L_j}(-a^{-1}t_1)\\
&=\tilde{x}_{-2L_j}(-a^{-1}t_1)\tilde{x}_{2L_j}(at_1^{-1})
\tilde{x}_{-2L_j}(-a^{-1}t_1)\tilde{w}_{2L_i}(at_2)\\
&=\tilde{w}_{-2L_j}(-a^{-1}t_1)\tilde{w}_{2L_i}(at_2).
\end{align*}
Similarly, for $t\in\KK^*$ keep using Lemma \ref{le:2}, we have
\begin{align*}
&\tilde{w}_{2L_i}(a)\tilde{w}_{L_i+L_j}(t_1,t_2)\tilde{w}_{2L_i}(a)^{-1}=\tilde{w}_{L_j-L_i}(t_2a^{-1},-t_1a^{-1}),\\
&\tilde{w}_{2L_i}(a)\tilde{w}_{L_i+L_j}(t,0)\tilde{w}_{2L_i}(a)^{-1}=\tilde{w}_{L_j-L_i}(0,-ta^{-1}),
\end{align*}
for $\gamma=\pm L_i\pm L_j(i\neq j)$ we have
\begin{align*}
\tilde{w}_{\gamma}(t_1,t_2)&=\tilde{w}_{-\gamma}(-t_1^{-1},-t_2^{-1}),\quad&\tilde{w}_{\gamma}(t,0)&=\tilde{w}_{-\gamma}(-t^{-1},0),\\
\tilde{w}_{\gamma}(0,t)&=\tilde{w}_{-\gamma}(0,-t^{-1}),\quad&\tilde{w}_{2L_i}(t)&=\tilde{w}_{-2L_i}(-t^{-1}).\\
\end{align*}
Hence we get the conclusion.
\end{proof}
Using a method similar to that in the proof of
\cite[7.7]{Steinberg}, we have
\begin{corollary}\label{cor:1}
Let $\tilde{W}$ be the subgroup of $\widetilde{G}$ generated by
$\{\tilde{w}_\gamma(u_1,u_2), \gamma\in\Phi,
(u_1,u_2)\in\KK^2\backslash 0$. Then $\tilde{W}$ is generated by
$\tilde{w}_{L_i-L_{i+1}}(u,0)$, $\tilde{w}_{L_i-L_{i+1}}(0,u)$ and
$\tilde{w}_{2L_{n}}(u)$ where $u\in\KK^*$.
\end{corollary}

\begin{lemma}\label{le:3} For $\gamma\in\Phi$, denote by
$\tilde{H}_{\gamma}$ the subgroup generated by
$\tilde{h}_{\alpha}(v_1,v_2)$, $(v_1,v_2)\in\KK^2\backslash 0$;
$\tilde{H}_{\gamma}^1$ the subgroup generated by
$\tilde{h}_{\alpha}(v,0)$ and $\tilde{H}_{\gamma}^2$ the subgroup
generated by $\tilde{h}_{\alpha}(0,v)$, $v\in\KK^*$. Let $\tilde{H}$
be he subgroup generated by $\{\tilde{H}_{\gamma}, \gamma\in\Phi\}$.
Then
\begin{itemize}
\item[1]$\tilde{H}_{\gamma}^\delta,\gamma\in\Phi$, $\delta=1,2$ is normal in
$\tilde{H}$, and $\tilde{H}$ is normal in $\tilde{W}$.

\item[2]$\tilde{H}$ normalizes each $\tilde{x}_\gamma$, and
hence $\tilde{x}^+$ which is generated by
$\tilde{x}_\beta(\beta\in\Phi^+)$.

\item[3]$\tilde{H}=\bigl(\prod_{\beta=L_i-L_{i+1}}(\tilde{H}_{\beta}^1\tilde{H}_{\beta}^2)\bigl)\cdot\tilde{H}_{2L_n} $.
\end{itemize}
\end{lemma}
\begin{proof}
The statements (i) and (ii) are clear from Lemma \ref{le:2}. For
(iii) we use Corollary \ref{cor:1} and a method similar to that in
the proof of \cite[7.7]{Steinberg}.
\end{proof}
An important step towards proof of Theorem \ref{th:5} is
\begin{proposition}
$\ker(\pi_1)\subseteq Z(\widetilde{G})\subseteq\tilde{H}$.
\end{proposition}
\begin{proof}
Step 1, we prove
$\tilde{W}\tilde{x}^+\tilde{W}\subseteq\tilde{x}^+\tilde{W}\tilde{x}^+$.

Denote by $\tilde{w}_{L_i-L_{i+1}}^1$ the subgroup generated by
$\tilde{w}_{L_i-L_{i+1}}(u,0)$, $\tilde{w}_{L_i-L_{i+1}}^2$ the
subgroup generated by $\tilde{w}_{L_i-L_{i+1}}(0,u)$ and
$\tilde{w}_{2L_n}$ the subgroup generated by $\tilde{w}_{2L_n}(u)$
where $u\in\KK^*$. Then by Corollary \ref{cor:1}, it is enough to
prove for any $w\in\tilde{W}$
$$w\tilde{x}^+\tilde{w}_{L_i-L_{i+1}}^\delta\subseteq
\tilde{x}^+\tilde{W}\tilde{x}^+\text{ and }
w\tilde{x}^+\tilde{w}_{2L_n}\subseteq
\tilde{x}^+\tilde{W}\tilde{x}^+, \delta=1,2.$$ We write
$\tilde{x}^+=\tilde{x}_{L_i-L_{i+1}}^1\tilde{x}'$ where
$\tilde{x}'=\prod\tilde{x}_\beta\tilde{x}_{L_i-L_{i+1}}^2$,
$\beta\in\Phi^+,\beta\neq L_i-L_{i+1}$.

If $w\tilde{x}_{L_i-L_{i+1}}^1w^{-1}\subseteq \tilde{x}^+$, we have
\begin{align*}
w\tilde{x}^+\tilde{w}_{L_i-L_{i+1}}^1&=w\tilde{x}_{L_i-L_{i+1}}^1\tilde{x}'\tilde{w}_{L_i-L_{i+1}}^1
\\
&\subseteq w\tilde{x}_{L_i-L_{i+1}}^1\tilde{w}_{L_i-L_{i+1}}^1\tilde{x}^+\\
&\subseteq(w\tilde{x}_{L_i-L_{i+1}}^1w^{-1})w\tilde{w}_{L_i-L_{i+1}}^1\tilde{x}^+\\
&\subseteq \tilde{x}^+\tilde{W}\tilde{x}^+.
\end{align*}

If $w\tilde{x}_{L_i-L_{i+1}}^1w^{-1}\subseteq \tilde{x}^-$, for any
$v\in\tilde{x}_{L_i-L_{i+1}}^1$, there is
$u\in\tilde{x}_{L_{i+1}-L_{i}}^1$ such that $vuv=w'\in
\tilde{w}_{L_i-L_{i+1}}^1$, thus we have
\begin{align*}
wv\tilde{w}_{L_i-L_{i+1}}^1&=ww'v^{-1}u^{-1}\tilde{w}_{L_i-L_{i+1}}^1\\
&\subseteq ww'v^{-1}\tilde{w}_{L_i-L_{i+1}}^1\tilde{x}^+,\\
&\subseteq
w\tilde{x}_{L_{i+1}-L_{i}}^1w'\tilde{w}_{L_i-L_{i+1}}^1\tilde{x}^+\\
&\subseteq \tilde{x}^+ww'\tilde{w}_{L_i-L_{i+1}}\tilde{x}^+.
\end{align*}
It follows that $$w\tilde{x}^+\tilde{w}_{L_i-L_{i+1}}^1\subseteq
\tilde{x}^+\tilde{W}\tilde{x}^+.$$
 The proof of $w\tilde{x}^+\tilde{w}_{L_i-L_{i+1}}^2\subseteq
\tilde{x}^+\tilde{W}\tilde{x}^+$ and
$w\tilde{x}^+\tilde{w}_{2L_n}\subseteq
\tilde{x}^+\tilde{W}\tilde{x}^+$ follows the same manner. Thus we
finished the first step.

Step 2, we prove $\ker(\pi_1)\subseteq\tilde{H}$.

Since $\widetilde{G}$ is generated by $\tilde{x}^+$ and $\tilde{W}$
and $\tilde{x}^+\cdot\tilde{x}^+\subseteq\tilde{x}^+$, by conclusion
of Step 1, we have $\widetilde{G}=\tilde{x}^+\tilde{W}\tilde{x}^+.$
If $\pi_1(x_1w'x_2)=e$, where $x_1,x_2\in \tilde{x}^+$ and
$w'\in\tilde{W}$, one has $\pi_1(w')=\pi_1(x_1^{-1}x_2^{-1})$. Since
$\pi_1(x_1^{-1}x_2^{-1})$ is of the following form
$$ \begin{pmatrix}A_1 & A_2 \\
0 & A_3 \\
 \end{pmatrix},$$
  where $A_1, A_2, A_3$ are $n\times n$ matrices with $A_1$ unipotent
  upper triangular and $A_3$ unipotent
  lower triangular. It follows immediately that $\pi(w')=e$ and
  $x_1=x_2^{-1}$ by uniqueness of such an expression\cite[p.24]{Steinberg2}. Hence by Lemma \ref{le:2} we have
\begin{align*}
x_1w'x_1^{-1}&=w'(w'^{-1}x_1w')x_1^{-1}\\
&=w'x_1x_1^{-1}\\
&=w'.
\end{align*}
If $\pi_1(w')=e$, then by similar arguments given by Steinberg
\cite[p.31 Lemma 22]{Steinberg2}, we have $w'\in\tilde{H}$.

Step 3, we prove $Z(\widetilde{G})\subseteq\tilde{H}$.

It is sufficient to prove that $Z(G)\subseteq H$. It is quite easy
to see that this is indeed so. Hence the proposition is proved.
\end{proof}
\begin{remark}In fact, by similar arguments given by Steinberg
\cite[p.31 Theorem 10]{Steinberg2}, or by similar explicit
calculations in \cite[p.186]{Silvester},
 we can show $(\widetilde{G},\pi_1)$ is a universal central extension.
\end{remark}
We now consider the conditions under which $\tilde{h}\in\tilde{H}$
is in the kernel of $\pi_1$.
\begin{lemma}\label{le:23}
\begin{align*}
&(1)  \ker(\pi_1)\cap \tilde{H}_{r}^\delta=\ker(\pi_1)\cap
\tilde{H}_{L_1-L_2}^1, \quad r=\pm L_i\pm
L_j(i\neq j),\delta=1,2,\\
&(2) \ker(\pi_1)\cap \tilde{H}_{2L_n}=\ker(\pi_1)\cap
\tilde{H}_{L_1-L_2}^1.\\
&(3) \ker(\pi_1)=\ker(\pi_1)\cap \tilde{H}_{L_1-L_2}^1.
\end{align*}
\end{lemma}
\begin{proof}(1) (2) Since these simple roots belong to the same orbit under the Weyl
group, an argument similar to one in \cite[Lemma 8.2]{Moore},  shows
that $\ker(\pi_1)\cap \tilde{H}_{r}^\delta\subseteq \ker(\pi_1)\cap
\tilde{H}_{L_1-L_2}^1$ for all roots $r=\pm L_i\pm L_j(i\neq j)$ and
$\ker(\pi_1)\cap \tilde{H}_{2L_n}=\ker(\pi_1)\cap
\tilde{H}_{L_1-L_2}^1$. This proves (1) and (2).

(3) For any $h\in \tilde{H}$, by Lemma \ref{le:3}, $h$ can be
written as $$h=h_1^1h_2^1\dots h_{n-1}^1h_1^2h_2^2\dots
h_{n-1}^2h_0$$ where $h_i^\delta\in
\tilde{H}_{L_{i}-L_{i+1}}^\delta(i\leq n-1)$, $\delta=1,2$ and
$h_0\in\tilde{H}_{2L_n}$.

If $\pi_1(h)=e$, notice
\begin{align*}
&\pi_1(\tilde{H}_{L_{i}-L_{i+1}}^1)=\text{diag}\left(a_i,(a^{-1})_{i+1}\right),a\in\KK^*\\
&\pi_1(\tilde{H}_{L_{i}-L_{i+1}}^2)=\text{diag}\left(a_{i+n},a^{-1}_{i+n+1}\right),a\in\KK^*,
\end{align*}
we have $\pi_1(h_i^\delta)=e$ for $0\leq i\leq n-1$, $\delta=1,2$
and $\pi_1(h_0)=e$. By (1) and (2) we get the conclusion.
\end{proof}

For $t_1,t_2\in\RR^*$, we define:
\begin{align*}
&\{t_1,t_2\}=\tilde{h}_{L_1-L_2}(t_1,0)\tilde{h}_{L_1-L_2}(t_2,0)\tilde{h}_{L_1-L_2}(t_1\cdot
t_2,0)^{-1}.
\end{align*}
Now in exactly the same manner as the proof in the appendix of
\cite{Moore}, we prove that these $\{t_1,t_2\}$'s, $\delta=1,2$
satisfy the conditions
\begin{lemma}\label{le:24}
\begin{align*}
\{t_1, t_2\}&=\{t_2,t_1\}^{-1}\qquad \forall t_1,t_2\in \KK^*,\\
\{t_1, t_2\cdot t_3\}&=\{t_1,t_2\}\cdot\{t_1,t_3\}\qquad \forall t_1, t_2,t_3\in \KK^*,\\
\{t_1\cdot t_2,
t_3\}&=\{t_1,t_3\}\cdot\{t_2,t_3\}\qquad \forall t_1, t_2,t_3\in \KK^*,\\
\{t,1-t\}&=1\qquad \forall t\in \KK^*, t\neq 1,\\
\{t,-t\}&=1\qquad \forall t\in\KK^*.
\end{align*}
Hence we also define a symbol on $\KK^*$.
\end{lemma}

\subsection{Proof of Theorem \ref{th:5}}

Notice for $t\in\KK^*$
$$\pi_1(\tilde{h}_{L_1-L_2}(t,0))=\diag\left(t_1,(t^{-1})_2\right).$$
Then
\begin{align*}
\ker(\pi_1)\cap \tilde{H}_{L_1-L_2}^1=\{\prod_i
\tilde{h}_{L_1-L_2}(t_i,0)\mid \text{ \emph{with} }\prod_i t_i=1\}.
\end{align*}
It follows that $\ker(\pi_1)\cap \tilde{H}_{L_1-L_2}$ is generated
by elements
$$\tilde{h}_{L_1-L_2}(t_1,0)\tilde{h}_{L_1-L_2}(t_2,0)\tilde{h}_{L_1-L_2}(t_1t_2,0)^{-1},\text{ where }t_1,t_2\in\KK^*.$$
Hence it is a immediate result by (3) of Lemma \ref{le:23}.

\end{document}